\documentclass[reqno,11pt]{amsart}
\usepackage{tkz-euclide}
\usepackage{hyperref}
\usepackage{amsmath} 
\usepackage{amsfonts}
\usepackage{mathtools}
\usepackage{amssymb} 
\usepackage{mathrsfs} 
\usepackage{graphicx}  
\usepackage{cancel}
\usepackage{bm}
\usepackage{tikz} 
\usepackage{centernot}
\usepackage{soul}
\usetikzlibrary{svg.path} 
\usepackage{longtable}
\usepackage{makecell}

\usepackage{geometry}
\usetikzlibrary{calc} 
\usepackage{xcolor}  
\usetikzlibrary{patterns}
\usetikzlibrary{arrows.meta} 
\usepackage{amsthm} 
\usepackage{float}
\usepackage{enumitem}
\setenumerate{ ref={\normalfont(\roman*)},label={\normalfont(\roman*)},itemsep=3pt}
\usepackage[english]{babel}
\usepackage[font=footnotesize, labelfont=bf]{ caption}
\usepackage{ esint }
\usepackage{stackengine,scalerel,graphicx}
\stackMath

\definecolor{trueblue}{rgb}{0.0, 0.45, 0.81}
\definecolor{truegreen}{rgb}{0.13, 0.55, 0.13}

\theoremstyle{plain}
\begingroup
\newtheorem{theorem}{Theorem}[section]
\newtheorem{lemma}[theorem]{Lemma}

\newtheorem{proposition}[theorem]{Proposition}

\endgroup

\theoremstyle{definition}
\begingroup

\endgroup

\newcommand{\B}{\mb B^3}
\renewcommand{\tilde}{\widetilde}

\DeclareMathOperator{\dv}{div}

\numberwithin{equation}{section}
\newcommand{\N}{\mathbb{N}}

\newcommand{\R}{\mathbb{R}}
\newcommand{\C}{\mathbb{C}}

\renewcommand{\S}{{\mathbb{S}}}

\usepackage[normalem]{ulem}

\def \mb{\mathbb}

\begin{document}
	
\title
[
Some lifting and approximation properties for maps in $W^{1,2}(\B;\S^2)$]
{
Some lifting and approximation properties\\ for maps in $W^{1,2}(\B;\S^2)$}
	
\author[Andr\'e Guerra]{Andr\'e Guerra} 
\address[Andr\'e Guerra]{Department of Pure Mathematics and Mathematical Statistics, University of Cambridge, Wilberforce Rd, Cambridge CB3 0WB, UK}
\email{adblg2@cam.ac.uk}

\author[Xavier Lamy]{Xavier Lamy} 
\address[{Xavier Lamy}]{Institut de Mathématiques de Toulouse, Université de Toulouse, 118 route de Narbonne, F-31062 Toulouse Cedex 8, France}
\email{Xavier.Lamy@math.univ-toulouse.fr}	
	
\author{Konstantinos Zemas}
\address[Konstantinos Zemas]{Institute for Applied Mathematics, University of Bonn\\
Endenicher Allee 60, 53115 Bonn, Germany}
\email{zemas@iam.uni-bonn.de}
	

\begin{abstract}
Smooth maps $u\colon\mathbb B^3\to\mathbb S^2$ 
can be lifted to $\hat u\colon\mathbb B^3\to\mathbb S^3$ using the Hopf fibration $h\colon \mathbb S^3\to\mathbb S^2$ via the factorization $u=h\circ\hat u$.
In this note we characterize the $W^{1,2}$-maps which have this lifting property in terms of exactness of the pullback form $u^*\omega_{\mathbb S^2}$, 
and deduce a smooth approximation property preserving the constraint $u^*\omega_{\mathbb S^2}=d\eta$.
%
\end{abstract} 
	
\maketitle
\thispagestyle{empty}


\section{Introduction}
The Hopf fibration  $h\colon\S^3\to\S^2$ 
is a prominent example of a nontrivial $\S^1$-bundle of $\S^3$ over $\S^2$. 
It  maps a unit vector $(z,w)\in\S^3\subset\C^2$ to the associated complex line $[(z,w)]\in \C\mathbb P^1\approx \S^2$, where the last identification is by stereographic projection.

Any smooth map $u\colon\B\to\S^2$ admits a smooth  Hopf lift $\hat u\colon \B\to\S^3$,
that is, it can be factorized as  $u=h\circ\hat u$,
 see \textit{e.g.} \cite[\textsection{2}]{Hardt2003}.
In this note we answer two questions recently raised in \cite{riviere2026gauge} about the extension of this lifting property to weakly differentiable maps, 
and related approximation issues.

Let us start by remarking that, 
if $u\in W^{1,2}(\B;\S^2)$ admits a Hopf lift $\hat u\in W^{1,2}(\B;\S^3)$, 
then the $2$-form $u^*\omega_{\S^2}\in L^1(\B;\bigwedge^2(\R^3)^*)$ must be (weakly) exact.
The chain rule implies indeed
\[u^*\omega_{\S^2}=\hat u^*(h^*\omega_{\S^2})=2\hat u^*d\theta\,,\]
where  $\theta\in \Omega^1(\S^3)$ is such that $h^*\omega_{\S^2}=2d\theta$.
One may check that
 $\eta=2\hat u^*\theta\in L^2(\B;\bigwedge^1(\R^3)^*)$
 satisfies
\[d\eta=2\hat u^*d\theta \]
in the sense of distributions (\textit{e.g.}, by 
mollifying $\hat u$ and extending $\theta$ to $\R^4$), 
  and therefore one has $u^*\omega_{\S^2}=d\eta$.
Our main result states that the converse is also true, 
thereby positively answering \cite[Open Problem  8]{riviere2026gauge}.

\begin{theorem}\label{main_thm:lifting}
Let $u\in W^{1,2}(\B;\S^2)$ be such that there exists $\eta\in L^2(\B;\bigwedge^1(\R^3)^*)$ with 
\begin{equation}\label{eq:1_pullback_u_closed}
d\eta=u^\ast\omega_{\S^2} \quad \text{in the sense of distributions}\,.
\end{equation}
Then, there exists $\hat u\in W^{1,2}(\B;\S^3)$ such that 
\begin{equation}\label{eq:Hopf_lift}
u=h\circ \hat u\,,
\end{equation}
where $h\colon \S^3\to \S^2$ is the Hopf fibration.
\end{theorem}

It would be interesting to also investigate under which conditions, for  $1\leq p<2$, a map $u\in W^{1,p}(\B;\S^2)$ admits a Hopf lift $\hat u\in W^{1,p}(\B;\S^3)$ as in \eqref{eq:Hopf_lift}. 
In the different context of Riemannian coverings $\tilde{\mathcal N}\to\mathcal N$, 
\textit{e.g.}, $\R\to\S^1$ or $\S^2\to\mathbb{RP}^2$,
the question of lifting $\mathcal N$-valued maps to $\tilde{\mathcal N}$-valued maps
 in Sobolev spaces has been quite thoroughly 
studied, see  \cite{VanSchaftingen2025} and references therein.
A notable difference with the present setting is that coverings have discrete fibers, while the Hopf fibers  are circles.

This is naturally linked with the problem of smooth approximation in $W^{1,2}(\B;\S^2)$.
In fact, it is known \cite{bethuel1990characterization} that \eqref{eq:1_pullback_u_closed} is equivalent to the existence of $(u_k)_{k\in \N}\subset C^\infty(\B;\S^2)$ such that \[u_k\to u \ \text{strongly in } W^{1,2}(\B;\S^2)\,,\]
and our proof of Theorem~\ref{main_thm:lifting} is about choosing Hopf lifts $\hat u_k$ of $u_k$ with good convergence properties.
Here, it is natural to look for smooth approximations of the pair $(u,\eta)$ of \eqref{eq:1_pullback_u_closed}
preserving in addition the nonlinear constraint $d\eta=u^*\omega_{\S^2}$.
The existence of a Hopf lift a posteriori enables us to achieve this,
thus
positively answering  \cite[Open Problem 7]{riviere2026gauge}.

\begin{theorem}\label{main_thm:approximation}
Under the assumptions of Theorem~\ref{main_thm:lifting},
there exist $(u_k)_{k\in \N}\subset C^\infty(\B;\S^2)$ and $(\eta_k)_{k\in \N}\subset \Omega^1(\B)$ such that
\begin{align}\label{eq:eta_strong_approximation} 
\begin{split}
\mathrm{(i)}&\quad
d\eta_k=u_k^\ast\omega_{\S^2}\,,
\\
\mathrm{(ii)}&\quad
(u_k,\eta_k)\to (u,\eta) \ \text{strongly in } W^{1,2}(\B;\S^2)\times L^{2}\big(\B;\bigwedge\nolimits^1(\R^3)^*\big)\,.
\end{split}
\end{align}
\end{theorem}

\section{Preliminaries}\label{sec:preliminaries}
\subsection{Some basic facts involving differential forms}\label{subsec:diff_forms}
 We adopt an \textit{extrinsic viewpoint}, namely we consider $$u:=(u^1,u^2,u^3):\B\to \mb S^2 \subset \R^3\,,$$ 
where $\R^3$ is endowed with the standard Euclidean basis $(e_1,e_2,e_3)$, so that $du(x)\in \R^3\otimes \R^3$ is given in coordinates as
\[du(x)=\sum_{k=1}^3du^k(x)\otimes e_k
{=\begin{pmatrix}du^1\\[2pt]
du^2\\[2pt]
du^3
\end{pmatrix}
}
\,.\] 

The volume form on $\mb S^n\subset\R^{n+1}$ is given by 
\begin{equation}\label{eq:omega_coordinate_expression}
\omega_{\mb S^n} = \sum_{j=1}^{n+1} (-1)^{j-1} x^j dx^1\wedge \dots \wedge \widehat{dx^j} \wedge \dots \wedge dx^{n+1}\,,
\end{equation}
where\, $\widehat{\cdot}$\, denotes that the corresponding term is omitted.
Thus, for $n=2$, its pull-back by $u$ 
is 
\begin{align}\label{eq:omega_coordinates}
\begin{split}	
u^*\omega_{\S^2}&=u^1du^2\wedge du^3+u^2du^3\wedge du^1+u^3du^1\wedge du^2\\
&=\frac{1}{2}u\cdot du\wedge du=\sum_{1\leq j<\ell\leq 3}u\cdot(\partial_j u\times \partial_\ell u)\,dx^j\wedge dx^\ell\,,	
\end{split}
\end{align}
where $\cdot$ and $\times$ denote the Euclidean inner and outer products  respectively. Since $|u|= 1$, this implies that 
\begin{align}\label{eq:modulus_ofu_omega}
|u^*\omega_{\S^2}|^2=\sum_{1\leq j<\ell \leq 3}|\partial_j u\times \partial_\ell u|^2\,.
\end{align}
We infer in particular the inequality
\begin{align}\label{eq:am_gm-ineq}
|u^*\omega_{\S^2}|\leq \frac 12 |du|^2\quad \mathcal{L}^3\text{-a.e. on }\B\,,
\end{align}
valid for any $u\in W^{1,2}(\B;\S^2)$.
We also note from the expression in the second line of \eqref{eq:omega_coordinates}, 
that $u^\ast\omega_{\S^2}$ can be identified,
via Hodge duality, with the vector field
\begin{equation}\label{eq:Hodge_dual_of_u_pullback_omega}
D(u):=\big(u\cdot(\partial_2u\times\partial_3u),u\cdot(\partial_3u\times\partial_1u), u\cdot(\partial_1u\times\partial_2u)\big)\,,
\end{equation}
so that, in the sense of distributions
\begin{equation}\label{eq:dpullback_divergence}
d(u^*\omega_{\S^2})=\dv (D(u))\,\mathrm{d}x\,,
\end{equation}
\textit{i.e.}, (distributional) closedness of the 2-form $u^\ast\omega_{\S^2}$ is equivalent to the vector field $D(u)$ being (distributionally) divergence-free.

We will also rely on the following form of the Hodge decomposition \cite[Corollary 5.6]{iwaniec1999nonlinear}:

\begin{proposition}\label{prop:particular_eta}
Let $1<p<\infty$. For any $\eta\in L^p\big(\mb B^3, \bigwedge^1(\R^3)^*\big)$, there exists a unique representative $\eta' \in L^p\big(\mb B^3,\bigwedge^1 (\R^3)^*\big)$ such that 
\begin{equation}\label{eq:choice_of_eta_'}
d \eta'=d \eta,\qquad d^* \eta'=0, \qquad \eta'_N=0
\end{equation}
in the weak sense. Precisely, the last two identities in \eqref{eq:choice_of_eta_'} are equivalent to
\begin{equation}
\label{eq:trace_def}
\int_{\mb B^3} \langle \eta', d\psi\rangle = 0 \quad \text{for all } \psi \in C^\infty(\mb B^3).
\end{equation}
\end{proposition}

\subsection{The Hopf map}\label{subsec:hopf}
%
Denoting by $\pi$ the inverse stereographic projection
\begin{equation*}
\pi\colon \C\cup\{\infty\} \to \S^2\,, \quad  \pi(z):=\left(\frac{2z}{|z|^2+1},\frac{|z|^2-1}{|z|^2+1}\right)\,,
\end{equation*}
the Hopf map $h:\S^3\subset \C^2\to \S^2$ is given by
\begin{equation}\label{def:Hopf}
h(z,w):=\pi(z/w)=(2z\bar w, |z|^2-|w|^2)\,,
\end{equation}
where the last equality holds true because $|z|^2+|w|^2=1$ for $(z,w)\in \S^3$.  
The 1-form on $\S^3\subset\R^4$ given by
\begin{equation}\label{eq:alpha_exp}
\theta:=-x^2dx^1+x^1dx^2-x^4dx^3+x^3dx^4\,,
\end{equation}
satisfies
\begin{equation}\label{eq:alpha_hopf}
h^*\omega_{\S^2}=2d\theta=4(dx^1\wedge dx^2+dx^3\wedge dx^4)\,.
\end{equation}

It is  convenient to choose particular coordinates on $\S^3$ and $\S^2$, namely
\begin{align}\label{eq:local_param}
\begin{split}
&\quad [0, \pi/2]\times \S^1\times \S^1\ni (t,e^{i\varphi_1}, e^{i\varphi_2})\mapsto (e^{i\varphi_1}\sin t, e^{i\varphi_2}\cos t) \in \S^3\,,\\
&\quad [0, \pi/2]\times \S^1\ni (t, e^{i\varphi})\mapsto (e^{i\varphi} \sin 2t,-\cos 2t) \in \S^2\,.
\end{split}
\end{align}
In these coordinates, we have
\begin{align*}
g_{\S^3}&=(dt)^2 + \sin^2t\,(d\varphi_1)^2+\cos^2t\,(d\varphi_2)^2\,,\\
g_{\S^2}&= 4(dt)^2+(\sin^22t)\,(d\varphi)^2\,,
\end{align*}
and  the Hopf map takes the simple form
\begin{align*}
h\colon (t,e^{i\varphi_1}, e^{i\varphi_2})\mapsto (t, e^{i(\varphi_1-\varphi_2)})\,.
\end{align*}

One can choose an orthonormal frame $\{\tau_1, \tau_2, \tau_3\}$, defined at a chart point $(t, e^{i\varphi_1}, e^{i\varphi_2}) \in \S^3$ as
\begin{equation}\label{eq:special_S3_frame}
\tau_1:=\frac{\partial}{\partial \varphi_1}+\frac{\partial}{\partial \varphi_2}\,,\quad \tau_2:=\frac{\partial}{\partial t}\,,\quad \tau_3:=\cot t\frac{\partial}{\partial \varphi_1}-\tan t\frac{\partial}{\partial \varphi_2}\,,
\end{equation}
so that $\tau_1$ is the \textit{fundamental vertical vector ﬁeld}, \textit{i.e.}, the generating vector ﬁeld of the $\S^1$-action $\S^1 \times \S^3\ni (e^{it}, (u,v))\to (e^{it}u,e^{it}v)\in \S^3$. With this choice, 
\begin{equation*}
dh(\tau_1)=0\,,\quad dh(\tau_2)=\frac{\partial}{\partial t}\,,\quad dh(\tau_3)=\frac{2}{\sin 2t}\frac{\partial}{\partial \varphi}\,,
\end{equation*}
where $\varphi:=\varphi_1-\varphi_2$. 
The pair $\{f_1,f_2\}:=\big\{\frac{1}{2}\frac{\partial}{\partial t},\frac{1}{\sin 2t}\frac{\partial}{\partial \varphi}\big\}$ is an orthonormal frame on $\S^2$, and in these coordinates the differential of the Hopf map $dh\colon T\S^3\to T\S^2$ has matrix representation 
\begin{equation}\label{eq:matrix_dh}
dh=\begin{pmatrix}
0&2&0\\
0&0&2
\end{pmatrix}\,.
\end{equation}

\subsection{Known approximation and lifting results}\label{subsec:Known_approx_lifting}

\begin{theorem}\label{thm:bethuel}
A map $u\in W^{1,2}(\B;\mb S^2)$ is in the strong $W^{1,2}$-closure of $C^\infty(\B;\mb S^2)$ if and only if $u^*\omega_{\mb S^2}$ is distributionally closed. Moreover, any map in $W^{1,2}(\B;\mb S^3)$ is in the strong $W^{1,2}$-closure of $C^\infty(\B;\mb S^3)$.
\end{theorem}

\begin{proof} The first assertion is a direct reformulation of \cite[Theorem~1]{bethuel1990characterization}, via the identification of the 2-form $u^\ast\omega_{\S^2}$ with the vector-field $D(u)$ of \eqref{eq:Hodge_dual_of_u_pullback_omega} and the subsequent identity \eqref{eq:dpullback_divergence}. The second assertion follows directly from \cite[Theorem 1]{bethuel1988density}.
\end{proof}

Thanks to the fact that any $\mathbb S^1$-bundle over $\B$ is trivial, 
every map $u\in C^\infty(\B;\S^2)$ admits a Hopf lift $\hat u\in C^\infty(\B;\S^3)$ such that $u=h\circ\hat u$.
Moreover, all possible lifts are classified according to their gauge  $\eta=2 \hat u^*\theta$, where $\theta$ is the $1$-form defined in \eqref{eq:alpha_exp}.

\begin{lemma}[{\cite[Lemma~2.1]{Hardt2003}}]
\label{l:smoothlift}
For any $u\in C^\infty(\B;\S^2)$ 
and  $\eta\in \Omega^1(\B)$ with $u^*\omega_{\S^2}=d\eta$, 
there exists a unique (modulo $\S^1$) Hopf lift $\hat u\in C^\infty(\B;\S^3)$ such that 
\begin{align}\label{eq:hatu_alpha}
u=h\circ\hat u
\quad
\text{and}
\quad
\eta= 2\hat u^*\theta\,.
\end{align}
\end{lemma}

Moreover, the chain rule relates the differentials $du$ and $d\hat u$ to
the gauge $\eta$. We note here that this relationship  is valid as soon
as $u$ and $\hat u$ are 
weakly differentiable and record in particular the following formula.

\begin{lemma}\label{lem:lift}	
For any $u\in W^{1,1}(\B;\S^2)$ and $\hat u\in W^{1,1}(\B;\S^3)$ such that 
$u=h\circ\hat u$, 
we have
\begin{equation}\label{eq:lift}
|d \hat u|^2 =\frac 14   |\eta|^2 + \frac 14 |du|^2\quad \mathcal{L}^3
\text{-a.e. on }\B\,,
\end{equation}
where $\eta\in L^1(\B;\bigwedge^1 (\R^3)^*)$ is given by $\eta=2\hat u^*\theta$,
for $\theta\in \Omega^1(\S^3)$  defined in \eqref{eq:alpha_exp}.
\end{lemma}

\begin{proof} 
The coefficients in \eqref{eq:lift} are actually different in \cite{Hardt2003} due to a different choice of normalization for the metric on $\mathbb S^2$, see \cite[Lemma 3.6]{guerra2025global}.
For the readers' convenience we reproduce the proof here.
In the orthonormal frame $\lbrace \tau_1,\tau_2,\tau_3\rbrace$ on $T\S^3$, the $1$-form $\theta$ of \eqref{eq:alpha_exp} corresponds to the scalar product against $\tau_1$, cf. \eqref{eq:special_S3_frame}.
Thus, the condition $\eta =2\hat u^*\theta$  and \eqref{eq:local_param}, \eqref{eq:special_S3_frame}, 
imply 
$\eta=2\tau_1(\hat u)\cdot d\hat u$, hence
\begin{equation}\label{eq:beta_norm}
|\eta|^2=4|\tau_1\cdot d\hat u|^2\,. 
\end{equation}	
Moreover, recalling the expression \eqref{eq:matrix_dh} of $dh$ in this orthonormal frame, the  chain rule gives
$
du=dh\circ d\hat u=2(\tau_2\cdot d\hat u)f_1 + 2(\tau_3\cdot d\hat u)f_2\,,
$
hence
\begin{equation*}
|du|^2=4|\tau_2\cdot d\hat u|^2 + 4 |\tau_3\cdot d\hat u|^2\,,
\end{equation*}
which, together with \eqref{eq:beta_norm},  implies \eqref{eq:lift}.
\end{proof}

\section{Proofs of the main theorems}\label{sec:proofs}

\begin{proof}[Proof of Theorem \ref{main_thm:lifting}]\label{pf:thm_lifting}
The proof follows in a sense the lines of proof of \cite[Theorem 3.3]{guerra2025global}. Using the assumption \eqref{eq:1_pullback_u_closed} and Theorem \ref{thm:bethuel} we can find a sequence of maps $(u_k)_{k\in \N}\subset C^\infty(\B;\mb S^2)$ such that 
\begin{equation}\label{eq:strong_W_1_2_approx}
u_k\rightarrow u \text{ strongly in } W^{1,2}(\B;\S^2) \text{ and } \mathcal{L}^3\text{-a.e., as }k\to\infty\,.
\end{equation} 
Since (by the smoothness of $u_k$), we have
\[du_k^\ast\,\omega_{\S^2}=u_k^\ast(d\omega_{\S^2})=0\,,\]
and $\B$ is simply connected, we can also fix smooth 1-forms 
$(\eta_k)_{k\in \N}\subset\Omega^1(\B)$
such that
\begin{align}\label{eq:dbeta_j}
d\eta_k= u^*_k\,\omega_{\mb S^2}\,, \qquad  d^*\eta_k=0\,, \qquad (\eta_k)_N=0\,,
\end{align} 
cf.\ Proposition \ref{prop:particular_eta}. By the Bourgain--Brezis elliptic estimates, see for instance \cite[Theorem 4.1]{brezis2007boundary}, we have
\begin{equation}\label{eq:elliptic_estimate}
\|\eta_k\|_{L^{3/2}(\B)}\leq C \|u^*_k\, \omega_{\mb S^2}\|_{L^1(\B)}\,,
\end{equation} 
for a universal constant $C>0$. We remark that, nevertheless, the critical exponent $3/2 = 1^*$ on the right-hand side is not needed for the argument below, and hence one could also appeal to more standard elliptic estimates.

Invoking Lemma~\ref{l:smoothlift}, we now fix lifts $(\hat u_k)_{k\in \N}\subset C^\infty(\B;\S^3)$ satisfying
\begin{equation}\label{eq:beta_j}
h\circ \hat u_k=u_k\,\ \text{and }\eta_k=2\hat u_k^*\theta\,.
\end{equation}
In view of \eqref{eq:am_gm-ineq} and \eqref{eq:strong_W_1_2_approx}-\eqref{eq:elliptic_estimate}, by passing to a non-relabeled subsequence, we have 
\begin{equation}\label{eq:weak_L_3_2_conv__forms}
\eta_k\rightharpoonup \tilde \eta \text{ weakly in } L^{3/2}\big(\B;\bigwedge\nolimits^1 (\R^3)^*\big) \text{ as }k\to\infty\,,
\end{equation}
{for some $\tilde \eta\in L^{3/2}\big(\B;\bigwedge\nolimits^1 (\R^3)^*\big)$.
By \eqref{eq:weak_L_3_2_conv__forms}, \eqref{eq:dbeta_j} and \eqref{eq:trace_def} (applied to $\eta_k$), 
we also see that $\tilde \eta_N=0$.}
By \eqref{eq:lift} applied to the triplet $(\hat u_k,\eta_k,u_k)$, together with \eqref{eq:strong_W_1_2_approx} and \eqref{eq:elliptic_estimate}, we have
\[\sup_{k\in\N}\|\hat u_k\|_{W^{1,3/2}(\B)}<+\infty\,,\]
and passing to a further non-relabeled subsequence, we obtain  $\hat u\in W^{1,3/2}(\B;\S^{3})$
so that 
\begin{equation}\label{eq:weak_W_3_2_conv__lifts}
\hat u_k\rightharpoonup \hat u\text{ weakly in } W^{1,3/2}\big(\B;\R^4\big)\text{ and pointwise } \mathcal{L}^3\text{-a.e. in } \B\,, \text{ as }k\to\infty\,.
\end{equation} 
Using now \eqref{eq:weak_L_3_2_conv__forms} and \eqref{eq:weak_W_3_2_conv__lifts}, we can pass to the limit in \eqref{eq:beta_j}, to infer that
\begin{equation}\label{eq:beta_equation}
h\circ \hat u=u\,,\ \ \tilde\eta = 2 \hat u^*\theta\ \ \mathcal{L}^3\text{- a.e.\ in } \B\,.
\end{equation} 
Indeed, the first identity in \eqref{eq:beta_equation} follows from the $\mathcal{L}^3$-a.e.\ convergence of $(\hat u_k,u_k)$ to $(\hat u, u)$, while for the second one therein, we observe the following. In view of \eqref{eq:alpha_exp}, the forms $\hat u_k^*\theta$ can be written in coordinates as second-order polynomials of the form $\hat u_k\bullet d\hat u_k$, where the latter notation indicates that each monomial in the expression is of first order separately in $\hat u_k$ and $d\hat u_k$. In particular, since 
\[
\hat u_k\rightarrow \hat u \ \text{ strongly in } L^3(\B;\R^3)\, \ \text{ and } d\hat u_k\rightharpoonup d\hat u \ \text{weakly in } L^{3/2}(\B; \R^3\otimes \R^3)\,, \text{ as } k\to \infty\,,\]
this product structure of weakly convergent against strongly convergent objects  justifies the second identity in \eqref{eq:beta_equation}, 
by taking the limit of the second identity of \eqref{eq:beta_j}. 

Further, for any $j,\ell\in\lbrace 1,2,3\rbrace$,
we have
\begin{align*}
u_k\cdot (\partial_j u_k\times\partial_\ell u_k)
&
=u_k\cdot (\partial_j u_k\times\partial_\ell u_k - \partial_j u\times\partial_\ell u)
+u_k\cdot (\partial_j u \times \partial_\ell u)\,.
\end{align*}
Since $|u_k|= 1$, by \eqref{eq:strong_W_1_2_approx} 
the first term in the right-hand side above converges to 0 in $L^1(\B)$, 
and the second term converges to 
$u\cdot(\partial_j u\times\partial_\ell u)$ in $L^1(\B)$ by dominated convergence.
Recalling the expression \eqref{eq:omega_coordinates} of $u^*\omega_{\S^2}$, 
we deduce the convergence 
\begin{align}\label{eq:conv_pullback}
u_k^*\omega_{\S^2}\to u^*\omega_{\S^2}
\quad\text{in } 
L^1\big(\B;\bigwedge\nolimits^2 (\R^3)^*\big)\,.
\end{align}
Therefore, taking also \eqref{eq:weak_L_3_2_conv__forms} into account, we may pass to the limit in \eqref{eq:dbeta_j} and deduce
\begin{equation*}\label{eq:dbeta_equation}
d\tilde\eta = u^*\omega_{\S^2}\,, \quad   d^*\tilde \eta=0\,, \quad \tilde \eta_N=0\,,
\end{equation*} 
in the sense of distributions. 
{We now note that, as $d \eta = u^*\omega_{\S^2}$ and $\eta \in L^2$, by Proposition \ref{prop:particular_eta}  we have also $\tilde \eta\in L^2$.	
Hence, by \eqref{eq:beta_equation} and \eqref{eq:lift}, we conclude
 that $\hat u\in W^{1,2}(\B;\S^3)$.}
\end{proof}

\begin{proof}[Proof of Theorem \ref{main_thm:approximation}]\label{pf:thm_approximation}
{By Theorem \ref{main_thm:lifting} there exists $\hat u\in W^{1,2}(\B;\S^3)$ satisfying \eqref{eq:1_pullback_u_closed} and \eqref{eq:Hopf_lift}. 
Applying Theorem \ref{thm:bethuel},
we find  $(\hat u_k)_{k\in \N}\subset C^\infty(\B;\S^3)$ such that }
\begin{equation}\label{eq:Bethuel_sequence}
\hat u_k\rightarrow \hat u \ \text{ strongly in } W^{1,2}(\B;\S^3)
\text{ and } \mathcal{L}^3\text{-a.e. in }\B\,.
\end{equation}
Let us now set 
\begin{equation}\label{eq:Hopf_projections}
u_k:=h\circ \hat u_k\in C^{\infty}(\B;\S^2)\,, \ \quad \zeta_k:=2\hat u_k^\ast \theta\in 
\Omega^1(\B)\,.
\end{equation}
Since $h$ is smooth on $\S^3$, the convergence \eqref{eq:Bethuel_sequence} and the lifting identity \eqref{eq:Hopf_lift} imply
\begin{equation}\label{eq:hat_u_k_good}
u_k=h\circ \hat u_k\rightarrow h\circ \hat u=u\, \ \text{strongly in } W^{1,2}(\B;\S^2)\,.
\end{equation} 
The explicit expression \eqref{eq:alpha_exp} of $\theta$ gives
\begin{equation}\label{eq:eta_k_monomial}
\zeta_k=2\hat u_k^\ast\theta=2\big(-\hat u_k^2d\hat u_k^1+\hat u_k^1d\hat u_k^2-\hat u_k^4d\hat u_k^3+\hat u_k^3d\hat u_k^4\big)=:\hat u_k\bullet d\hat u_k\,,
\end{equation} 
where, as in the proof of Theorem~\ref{main_thm:lifting},
the latter notation indicates that each monomial in the expression \eqref{eq:eta_k_monomial} is of first order separately in $\hat u_k$ and $d\hat u_k$. 
As for the proof of \eqref{eq:conv_pullback}, we write
\begin{align*}
\hat u_k\bullet d\hat u_k &
=\hat u_k \bullet (d\hat u_k -d\hat u) +  \hat u_k\bullet d\hat u\,.
\end{align*}
The fact that $|u_k|=1$ and the convergence \eqref{eq:Bethuel_sequence} 
ensure that the first term in the right-hand side converges to 0 in $L^2$, and the second term converges in $L^2$ to $\hat u\bullet d\hat u$ by dominated convergence.
Hence letting $k\to\infty$ in \eqref{eq:eta_k_monomial} gives
\begin{align}\label{eq:eta_k_ok}
\zeta_k\to 2\hat u^*\theta\quad\text{strongly in }L^2\big(\B;\bigwedge\nolimits^1(\R^3)^*\big)\,.
\end{align}
 Since 
\begin{equation*}
d(\eta-2\hat u^*\theta)=0\,,	
\end{equation*}
cf. again \eqref{eq:hatu_alpha}, \eqref{eq:alpha_hopf}, and where $\eta$ is as in \eqref{eq:1_pullback_u_closed}, by a standard mollification argument we can also find $(\tilde \zeta_k)_{k\in \N}\subset \Omega^1(\B)$ such that 
\begin{align}\label{eq:tilde_zeta_k_ok}
d\tilde \zeta_k=0\,,\quad	\tilde \zeta_k\to \eta-2\hat u^*\theta\quad\text{strongly in }L^2\big(\B;\bigwedge\nolimits^1(\R^3)^*\big)\,.
\end{align}
Finally, setting $\eta_k:=\zeta_k+\tilde\zeta_k$, \eqref{eq:eta_strong_approximation}(i) follows from \eqref{eq:Hopf_projections} and the first property in \eqref{eq:tilde_zeta_k_ok}, while \eqref{eq:eta_strong_approximation}(ii) from  \eqref{eq:eta_k_ok} and \eqref{eq:tilde_zeta_k_ok}.
\end{proof}

\section*{Acknowledgements}

AG acknowledges the support of the Royal Society through a Newton International Fellowship. XL was supported by 
the ANR project ANR-22-CE40-0006.  KZ was funded by the Deutsche Forschungsgemeinschaft (DFG, German Research Foundation) – CRC 1720 – 539309657.

\bibliographystyle{acm}
\bibliography{ref_hopf}

@article{brezis2007boundary,
	author = {Brezis, Ha{\"\i}m and Van Schaftingen, Jean},
	journal = {Calculus of Variations and Partial Differential Equations},
	number = {3},
	pages = {369--388},
	publisher = {Springer},
	title = {Boundary estimates for elliptic systems with L1--data},
	volume = {30},
	year = {2007}}

@article{VanSchaftingen2025,
 author = {van Schaftingen, Jean},
 title = {Lifting of fractional {Sobolev} mappings to noncompact covering spaces},
 fjournal = {Annales de l'Institut Henri Poincar{\'e} C. Analyse Non Lin{\'e}aire},
 journal = {Ann. Inst. Henri Poincar{\'e} C, Anal. Non Lin{\'e}aire},
 issn = {0294-1449},
 volume = {42},
 number = {1},
 pages = {41--84},
 year = {2025},
 language = {English},
 doi = {10.4171/AIHPC/98},
 zbMATH = {7988204},
 Zbl = {1570.46033}
}

@article{Hardt2003,
	author = {Hardt, R. and Rivi{\'e}re, T.},
	fjournal = {Annali della Scuola Normale Superiore di Pisa. Classe di Scienze. Serie V},
	issn = {0391-173X},
	journal = {Ann. Sc. Norm. Super. Pisa, Cl. Sci. (5)},
	language = {English},
	number = {2},
	pages = {287--344},
	title = {Connecting topological {Hopf} singularities},
	url = {https://eudml.org/doc/84503},
	volume = {2},
	year = {2003},
	zbl = {1150.58004},
	zbmath = {5019611}}

@article{iwaniec1999nonlinear,
	author = {Iwaniec, T. and Scott, C. and Stroffolini, B.},
	doi = {10.1007/BF02505905},
	fjournal = {Annali di Matematica Pura ed Applicata. Serie Quarta},
	issn = {0373-3114},
	journal = {Ann. Mat. Pura Appl. (4)},
	language = {English},
	pages = {37--115},
	title = {Nonlinear {Hodge} theory on manifolds with boundary},
	volume = {177},
	year = {1999},
	zbl = {0963.58003},
	zbmath = {1472046}}

@article{riviere2026gauge,
	author = {Rivi{\`e}re, Tristan},
	journal = {arXiv preprint arXiv:2604.06026},
	title = {Gauge Symmetry Breaking in the Asymptotic Analysis of Self Dual Yang-Mills-Higgs $ SU (2) $ Monopoles},
	x-fetchedfrom = {Google Scholar},
	year = {2026}}

@article{bethuel1990characterization,
	author = {Bethuel, F.},
	doi = {10.1016/S0294-1449(16)30292-X},
	fjournal = {Annales de l'Institut Henri Poincar{\'e}. Analyse Non Lin{\'e}aire},
	issn = {0294-1449},
	journal = {Ann. Inst. Henri Poincar{\'e}, Anal. Non Lin{\'e}aire},
	language = {English},
	number = {4},
	pages = {269--286},
	title = {A characterization of maps in {{\(H^ 1(B^ 3,S^ 2)\)}} which can be approximated by smooth maps},
	url = {https://eudml.org/doc/78224},
	volume = {7},
	year = {1990},
	zbl = {0708.58004},
	zbmath = {4163760}}

@article{guerra2025global,
	author = {Guerra, Andr{\'e} and Lamy, Xavier and Zemas, Konstantinos},
	journal = {arXiv preprint arXiv:2507.10686},
	title = {Global minimality of the Hopf map in the Faddeev-Skyrme model with large coupling constant},
	x-fetchedfrom = {Google Scholar},
	year = {2025}}

@article{bethuel1988density,
	author = {Bethuel, Fabrice and Zheng, Xiaomin},
	journal = {Journal of functional analysis},
	number = {1},
	pages = {60--75},
	publisher = {Elsevier},
	title = {Density of smooth functions between two manifolds in Sobolev spaces},
	volume = {80},
	x-fetchedfrom = {Google Scholar},
	year = {1988}}

\end{document}